\documentclass[12pt,a4paper]{amsart}
\usepackage[top=1.15in, bottom=1.15in, left=1.17in, right=1.17in]{geometry}
\usepackage{amsthm, amsmath, amssymb, amscd, latexsym, multicol, verbatim, enumitem, graphicx,xy, color, stmaryrd}
\usepackage{tikz}
\usepackage{float}
\usepackage{hyperref}
\hypersetup{colorlinks=true, pdfstartview=FitV, linkcolor=blue, citecolor=blue, urlcolor=blue, breaklinks=true}
\usepackage[english]{babel}
\usepackage[T1]{fontenc}
\usepackage[utf8]{inputenc}
\usepackage{mathrsfs}
\usepackage{cleveref}
\usepackage{tikz-cd}
\usepackage{quiver}
\usepackage{adjustbox}
\usetikzlibrary{arrows,backgrounds,calc,trees,hobby}
\usepackage{csquotes}
\usepackage{graphicx}
\usepackage{tabularx,booktabs}
\usepackage[clockwise]{rotating}
\usepackage{xspace}
\usepackage[ruled, vlined, linesnumbered]{algorithm2e}

\usepackage[backend=bibtex,style=alphabetic,doi=false,isbn=false,url=false,giveninits=true,maxbibnames=50]{biblatex}
\newcounter{cnt}
 \makeatletter
\def\mydggeometry{\makeatletter\dg@YGRID=1\dg@XGRID=20\unitlength=0.003pt\makeatother}
\makeatother

\newtheorem{theorem}{Theorem}[section]
\newtheorem{mainresult}{Theorem}
\newtheorem*{theorem*}{Theorem}
\newtheorem{lemma}[theorem]{Lemma}
\newtheorem{corollary}[theorem]{Corollary}
\newtheorem{proposition}[theorem]{Proposition}

\newtheorem*{conjecture*}{Conjecture}

\theoremstyle{definition}
\newtheorem{definition}[theorem]{Definition}
\newtheorem{example}[theorem]{Example}
\newtheorem{remark}[theorem]{Remark}

\newcommand{\Z}{\mathbb{Z}}
\newcommand{\C}{\mathbb{C}}

\newcommand{\N}{\mathbb{N}}

\pgfdeclarelayer{background}
\pgfsetlayers{background,main}

\begin{document}
\author[Ahmad]{Ibrahim Ahmad}
\title[Flat degenerations of flag supermanifolds for basic superalgebras]{Flat degenerations of flag supermanifolds for basic Lie superalgebras}

\address{Chair of Algebra and Representation Theory, RWTH Aachen University, Pontdriesch 10-16, 52062 Aachen, Germany}
\email{ahmad@art.rwth-aachen.de}

\maketitle

\begin{abstract}
  Motivated by bases of representations compatible with the PBW filtration for basic Lie superalgebras by Kus and Fourier, we generalise the construction of degenerations of flag varieties via favourable modules to the super setup. In the classical setup, this method of degenerating flag varieties by Feigin, Fourier and, Littelmann relies on constructing bases of representations of Lie algebras such that in the coordinate ring of an embedded flag varietiy their multiplication can be identified with an affine semigroup modulo terms of higher degree. By killing off said terms of higher degree via a filtration construction, one gets a toric variety the embedded flag variety degenerates into. By adapting these techniques we provide a similar construction and discuss when one can get a degeneration into a toric supervariety, as defined by Jankowski.
\end{abstract}

\section{Introduction}
Degeneration methods have proven to be of great aid when verifying questions on specific properties of varieties. With a degeneration of a variety $X$, we mean a flat family over, say, $\mathbb{A}^{1}$ such that the general fibre over $t\neq0$ is isomorphic to our initial variety $X$. The special fibre over $t=0$ is isomorphic to a, hopefully, less complicated variety, call it $Y$. By imposing flatness, one can then verify several open properties for $X$ by studying such properties for $Y$ instead. Among these are notions like dimensions, Cohen-Macauleyness, smoothness, and normality. For a condensed list of some open properties, see \cite[Appendix E]{Görtz-Wedhorn}

Those degenerations with the special fibre being a toric variety are particularly interesting. One of the most prominent methods to systematically construct such degenerations of embedded projective varieties have been the notions of Newton-Okounkov bodies \cite{Newton-Okounkov} as well as degenerations into semitoric varieties via Seshadri Stratifications by Chiriv\`i, Fang and, Littelmann \cite{Seshadri-Stratification}.

For varieties arising from representation theory, there have also been many other methods of constructing these degenerations like \cite{Degeneration-Schubert, Degeneration-Spherical,Degeneration-Flag} by considering bases of coordinate rings with respect to an embedding that multiply in a controlled manner, some refer to as a straightening law for example in \cite{Seshadri-Stratification}. Among these methods are the notions of favourable modules by Feigin, Fourier and Littelmann \cite{Favourable} as well as their generalisation of birational sequences \cite{Birational_Sequence}. The latter is also implemented in the computer algebra system OSCAR \cite{Oscar-Monomial}.

In \cite{Kus-Fourier}, Kus and Fourier construct bases of certain representations of Lie Superalgebras and propose these could provide a similar construction in a supergeometric setting. Supergeometry can be regarded as an extension of algebraic geometry in the sense that the sheaf of commutative rings, or algebras, is being replaced by a sheaf of $\Z/2\Z$-graded algebras, referred to as superalgebras, obeying a signed version of commutativity incorporating commuting as well as anti-commuting variables. Within this context, one can also define Lie superalgebras as being $\Z/2\Z$-graded spaces with commutators satisfying signed versions of axioms for Lie algebras. These objects enjoy rising popularity due to their applications in theoretical physics but prove themselves to lie deeper in pure mathematics than initially anticipated.

In this paper, we will lift the degeneration via favourable modules of Feigin-Fourier-Littelmann \cite{Favourable} to the context of flag supermanifolds for basic Lie superalgebras. We briefly recall their result:

For a semisimple, simply connected complex algebraic group $G$ with corresponding Lie algebra $\mathfrak{
g}$, we consider a regular weight $\lambda\in P^{++}$. Then, we can embed the flag variety into the projectivized highest weight space $\mathbb{P}(V(\lambda))$ with the coordinate ring being the graded algebra
\begin{equation*}
    R(\lambda):=\bigoplus_{k\in\mathbb{N}_{0}}V(k\lambda)^*.
\end{equation*}
The multiplication is the dual of the Cartan embedding which is the following homomorphism of representations of $\mathfrak{g}$ sending the highest weight vector of $V((m+n)\lambda)$ to the tensor of the highest weight vectors of $V(m\lambda)$ and $V(n\lambda)$
\begin{align*}
    V((m+n)\lambda)&\rightarrow V(m\lambda)\otimes V(n\lambda)\\
    v_{(m+n)\lambda}&\mapsto v_{m\lambda}\otimes v_{n\lambda}
\end{align*}

Fix a triangular decomposition $\mathfrak{g}=\mathfrak{n}^{-}\oplus\mathfrak{h}\oplus\mathfrak{n}^{+}$, an ordered basis $(f_1,\dots,f_N)$ of $\mathfrak{n}^{-}$ and a monomial order on $\Z^N$. Then, Feigin, Fourier and Littelmann define a monomial $\mathbf{m}=(m_1,\dots,m_N)\in\Z^N$ to be \textit{essential} if
\begin{equation*}
    f^{\mathbf{m}}v_{\lambda}=f_1^{m_{1}}\dots f_N^{m_N}v_\lambda\not\in\mathrm{span}\langle f^{\mathbf{n}}v_\lambda\mid\mathbf{n}<\mathbf{m}\rangle
\end{equation*}
Denote the set of essential monomials for $V(\lambda)$ by $\mathrm{es}(\lambda,<)$. Then it is the case that if $\mathbf{m}\in\mathrm{es}(m\lambda,<)$ and $\mathbf{n}\in\mathrm{es}(n\lambda,<)$, then $\mathbf{m+n}\in\mathrm{es}((m+n)\lambda,<)$, i.e. they have a semigroup structure. If the semigroup
\begin{equation*}
    \Gamma(\lambda,<)=\bigcup_{k\in\mathbb{N}_0}\mathrm{es}(k\lambda,<)\times\{k\}
\end{equation*}
is generated by $\mathrm{es}(\lambda,<)\times\{1\}$, then they call the module $V(\lambda)$ \textit{favourable} and get a degeneration of $R(\lambda)$ into the semigroup algebra of $\Gamma(\lambda,<)$ corresponding to a toric variety. 
Prominent examples of bases yielding such a degeneration are those given by lattice points the FFLV polytopes for $\mathrm{SL}_{n+1}$ \cite{FFLV-A} and $\mathrm{Sp}_{2n}$ \cite{FFLV-C}.

Kus and Fourier constructed in \cite{Kus-Fourier} bases of representations of Lie superalgebras parametrised by lattice points of polytopes, similar to the FFLV polytope. In their paper, they raised the question if a degeneration construction similar to the one via favourable modules would work which is what we are going to address in this paper. 

We consider for a basic Lie superalgebra $\mathfrak{g}$ and a dominant integral weight $\lambda\in\mathfrak{h}^+$ the superalgebra
\begin{equation*}
    R(\lambda)=\bigoplus_{k\in\N_{0}}K_{\mathfrak{b}}(k\lambda)^*,
\end{equation*}
where $K_{\mathfrak{b}}(k\lambda)$ is referring to the finite-dimensional quotient of the Verma module $M_\mathfrak{b}(\lambda)$ and multiplication being again the dual of the morphism sending the highest weight vector to the tensor of highest weight vectors. 
This algebra can also be interpreted as the section ring of the flag supermanifold $\mathcal{G}/\mathcal{B}$ for an algebraic supergroup $\mathcal{G}$ and Borel subsupergroup $\mathcal{B}$ with corresponding Lie superalgebras $\mathfrak{g}$ and $\mathfrak{b}$. Namely we consider the line bundle $\mathcal{L}_{\lambda}$ for the weight $\lambda\in\mathfrak{h}^{*}$ \cite[Corollary 9]{Serg_Rep} and have
    \begin{equation*}
        R(\lambda)\cong\bigoplus_{k\in\N_{0}}\mathrm{H}^{0}(\mathcal{G}/\mathcal{B},\mathcal{L}_{k\lambda})
    \end{equation*}
We have the following result
\begin{mainresult}{(Corollary \ref{result1})}
    There exists a morphism of superschemes $\kappa:\mathrm{Spec}\mathcal{R}\rightarrow\mathbb{A}^{1|0}$ that is flat and the fibers are
    \begin{align*}
        \kappa^{-1}(a)\cong\begin{cases}
            \mathrm{Spec}\,R(\lambda),\hspace{1em}a\neq0\\
            \mathrm{Spec}\,\mathbb{C}[\xi^{I}x^{\mathbf{m}}v|(I,\mathbf{m})\in\mathrm{es}(K_{\mathfrak{b}}(\lambda),<)],\hspace{1em}a=0
        \end{cases}
    \end{align*}
\end{mainresult}
Further, we also have that in the special cases of those bases given by Kus and Fourier in \cite{Kus-Fourier} one gets a toric supervariety as defined by Jankowski in \cite{Jankowski2025}.
\begin{mainresult}{(Corollary \ref{Kus-Fourier-Toric})}
    Let $\mathfrak{g}$ be a Lie superalgebra of type I, $\mathfrak{osp}(1|2n)$ or one of the basic exceptionals. Let $\lambda\in\mathfrak{h}^{*}$ be a dominant integral weight that is typical such that $k\lambda$ is also typical for all $k\in\N$. Denote by $P(\lambda)$ the corresponding lattice polytope from \cite{Kus-Fourier} and $S(\lambda)$ its lattice points.
    Then the superscheme $\mathrm{Spec}\, A$ for
    \begin{equation*}
        A:=\C[\xi^{I}x^{\mathbf{m}}v|(I,\mathbf{m})\in S(\lambda)]
    \end{equation*}
    is a faithful toric supervariety.
\end{mainresult}

We are going to proceed as follows. In section \ref{sec-preliminaries}, we recall some basic notions of Lie superalgebras. In section \ref{sec-essential-monomials}, we define the superalgebra we are going to deal with. In section \ref{sec-basis-straight-laws}, we inspect the multiplication of this superalgbera more closely and retrieve a straightening law. Section \ref{sec-favourable} is devoted to the notion of a favourable module and retrieving an ideal representing it. In section \ref{sec-flat-family}, we are going to construct the flat family. Lastly, section \ref{sec-supergeom} states the result in a more geometric setting and provides a criterion for the special superalgebra generated by monomials to correspond to an affine toric supervariety and what possible torus actions are.

\noindent
\textbf{Acknowledgment:}
The author would like to thank his advisor, Ghislain Fourier, for suggesting the special and exciting question and Xin Fang for many helpful comments on previous drafts. The author also thanks Eric Jankowski and Alexander Sherman for several helpful discussions. The author is funded by the Deutsche Forschungsgemeinschaft (DFG, German Research Foundation) through \textit{Symbolic Tools in Mathematics and their Application} (TRR 195, project-ID 286237555).
\section{Preliminaries}\label{sec-preliminaries}
In this section, we are going to recall the notions on Lie superalgebras, which can be found in \cite{Fioresi,Cheng-Wang,Serg_Rep}.\smallskip

A \textit{supervector space} is a $\Z/2\Z$-graded vector space $V=V_{\overline{0}}\oplus V_{\overline{1}}$. 
We refer to its \textit{dimension} as the tuple $\dim V:=\dim V_{\overline{0}}|\dim V_{\overline{1}}$. 
For a homogeneous element $v\in V_{i}$ we refer to $|v|=i\in\{\overline{0},\overline{1}\}$ as its \textit{parity}. We call a linear map between supervector spaces $A:V\rightarrow W$ \textit{even} if it preserves the grading. If it flips the grading, then we call it \textit{odd}. Hence for supervector spaces $V,W$ we have that the set of of linear maps $\mathrm{Hom}_{\C}(V,W)$ is itself a supervector space. The dual space $V^*$ becomes a superspace by setting $V^*:=\mathrm{Hom}_{\C}(V,\C^{1|0}:=\C\oplus\{0\})$.\smallskip

A \textit{superalgebra} is a supervector space $A=A_{\overline{0}}\oplus A_{\overline{1}}$ such that the multiplication map is even, i.e. $A_{\overline{i}}A_{\overline{j}}\subseteq A_{\overline{i+j}}$. We call it \textit{commutative} if further
\begin{equation*}
    ab=(-1)^{|a||b|}ba
\end{equation*}
for homogenous elements $a,b\in A$.\smallskip

We call a supervector space $\mathfrak{g}$ a \textit{Lie superalgebra} if there exists a bilinear even map $[\cdot,\cdot]:\mathfrak{g}\times\mathfrak{g}\rightarrow \mathfrak{g}$ such that for all homogeneous $x,y,z\in\mathfrak{g}$
\begin{enumerate}
    \item{
        $[x,y]=-(-1)^{|x|\cdot|y|}[y,x]$
    }
    \item{
        $[x,[y,z]]=[[x,y],z]+(-1)^{|x|\cdot|y|}[y,[x,z]]$
    }
\end{enumerate}
We call $\mathfrak{g}$ \textit{simple} if it has no non-trivial $\mathbb{Z}/2\Z$-graded ideal. We further call it \textit{basic} if it is simple, $\mathfrak{g}_{\overline{0}}$ is a reductive Lie algebra and $\mathfrak{g}$ admits a non-degenerate invariant bilinear $(\cdot,\cdot)$, i.e.
\begin{equation*}
    ([x,y],z)=(x,[y,z]),\hspace{1em}(x,y)=(-1)^{|x|\cdot|y|}(y,x),\hspace{1em}\text{for homogenous }x,y\in\mathfrak{g}
\end{equation*}
as well as
\begin{equation*}
    (x,y)=0,\hspace{1em}\text{ for $x\in\mathfrak{g}_{\overline{0}}$, $y\in\mathfrak{g}_{\overline
    1}$}
\end{equation*}

Kac has classified basic Lie superalgebras.

\begin{theorem}[{\cite{Kac_Lie}}]
    The basic Lie superalgebras over $\mathbb{C}$ that are not purely even Lie algebras are
    \begin{enumerate}
        \item $\mathfrak{sl}(m|n)$, $n>m\geq 1$;
        \item $\mathfrak{psl}(n|n)$, $n\geq 2$;
        \item $\mathfrak{osp}(m|2n)$ $m,n\geq 1$ $(m,n)\neq(2,1), (4,1)$;
        \item $D(2, 1;\alpha)$, $\alpha\neq0,1$;
        \item $F(4)$;
        \item $G(3)$.
    \end{enumerate}
    The Lie superalgebras, $\mathfrak{sl}(n|m)$, $\mathfrak{psl}(n|n)$, $\mathfrak{\mathfrak{osp}}(2|2n)$ are referred to as type I. The other members of this list are referred to as type II.
\end{theorem}
\begin{remark}
    Some pieces of literature also declare $\mathfrak{gl}(m|n)$ to be basic and then of type I, see \cite{Cheng-Wang}.
\end{remark}
Next, we want to state the notions of roots and Borel subalgebras for Lie superalgebras.\smallskip

Let $\mathfrak{g}$ be a basic Lie superalgebra. We define a \textit{Cartan subalgebra} $\mathfrak{h}$ to be a Cartan subalgebra of the even part $\mathfrak{g}_{\overline{0}}$.\smallskip

For a given $\alpha\in\mathfrak{h}^{*}$, we define the root space
\begin{equation*}
    \mathfrak{g}_{\alpha}=\left\{x\in\mathfrak{g}\,\middle| \,[h,x]=\alpha(h)x\,\text{ for all }h\in\mathfrak{h}\right\}.
\end{equation*}
We call $\alpha$ a \textit{root} if $\mathfrak{g}_\alpha\neq\{0\}$ and denote the set of roots by $R$.
This yields us the decomposition
\begin{equation*}
    \mathfrak{g}=\mathfrak{h}\oplus\bigoplus_{\alpha\in R}\mathfrak{g}_{\alpha}
\end{equation*}
For basic Lie superalgebras, we have that $\dim\mathfrak{g}_\alpha=1|0$ or $\dim\mathfrak{g}_\alpha=0|1$ \cite[Theorem 1.8]{Cheng-Wang}, hence allowing us to define the set of even and odd roots via
\begin{equation*}
    R_{\overline{0}}=\left\{\alpha\in R\,\middle|\mathfrak{g}_{\alpha}\cap\mathfrak{g}_{\overline{0}}\neq\{0\}\right\},\hspace{1em}R_{\overline{1}}=\left\{\alpha\in R\,\middle|\mathfrak{g}_{\alpha}\cap\mathfrak{g}_{\overline{1}}\neq\{0\}\right\}.
\end{equation*}
This, in turn, leads us to the notions of positive systems and Borel subalgebras.\smallskip

Let $E$ be the real space spanned by the roots $R$. For a fixed total order $\geq$ on $E$ that is compatible with the real vector space structure on $E$, we define the \textit{positive system} $R^{+}$ as the set of all roots $\alpha\in R$ such that $\alpha>0$. The elements of $R^{+}$ are referred to as \textit{positive roots}. For such a positive system, we define the subset $\Delta\subseteq R^{+}$ of \textit{simple roots} as those that cannot be written as a sum of two positive roots.
Similarly, we also define the \textit{negative roots} to be the set $R^{-}$ of all roots $\alpha\in R$ such that $\alpha<0$.

For a given positive system $R^{+}\subseteq R$, we define the \textit{Borel subalgebra}
\begin{equation*}
    \mathfrak{b}=\mathfrak{h}\oplus\bigoplus_{\alpha\in R^{+}}\mathfrak{g}_{\alpha}=:\mathfrak{h}\oplus\mathfrak{n}^{+}
\end{equation*}
In turn, we also get the triangular decomposition
\begin{equation*}
    \mathfrak{g}=\mathfrak{n}^{-}\oplus\mathfrak{h}\oplus\mathfrak{n}^{+}=\bigoplus_{\alpha\in R^{-}}\mathfrak{g}_{\alpha}\oplus\mathfrak{b}
\end{equation*}
\begin{remark}
    Similarly to the even case, one can define a Weyl group and its action on Borel subalgebras. However, it is not the case that all Borel subalgebras are conjugate under the Weyl group action, see \cite[Section 1.3 and 1.4]{Cheng-Wang}
\end{remark}
We can form for $\lambda\in\mathfrak{h}^{*}$ the \textit{Verma module} with respect to a Borel subalgebra $\mathfrak{b}$ via
\begin{equation*}
    M_{\mathfrak{b}}(\lambda):=U(\mathfrak{g})\otimes_{U(\mathfrak{b})}C_{\lambda},
\end{equation*}
where $C_\lambda$ is the one-dimensional $\mathfrak{b}$-module with trivial action of $\mathfrak{n}^{+}$ and $\mathfrak{h}$ acting by the weight $\lambda$.
This module again has a unique simple quotient denoted by $L_{\mathfrak{b}}(\lambda)$. We also say that $\lambda$ is \textit{dominant integral} if $L_{\mathfrak{b}}(\lambda)$ is finite-dimensional. 
In the case of $\lambda$ being dominant integral, then we define the unique maximal finite dimensional quotient of $M_{\mathfrak{b}}(\lambda)$ as $K_{\mathfrak{b}}(\lambda)$.
\section{Essential monomials and multiplications}\label{sec-essential-monomials}
From now on, we fix a Borel subalgebra $\mathfrak{b}$ of $\mathfrak{g}$. In this section, we will introduce the superalgebra subject to our degeneration. The multiplication arises from the following result.
\begin{proposition}[Cartan mapping]\label{prop-cartan-embedding}
    Let $\lambda\in\mathfrak{h}^*$ be a dominant integral weight such that $k\lambda$ remains dominant integral for any $k\in\N$ and fix natural numbers $m,n\in\N$. Then, the map of $\mathfrak{g}$-modules defined on the highest weight vector given by
    \begin{align*}
        C_{m,n}:K_{\mathfrak{b}}((m+n)\lambda)&\rightarrow K_{\mathfrak{b}}(m\lambda)\otimes_{\C} K_{\mathfrak{b}}(n\lambda)\\
        v_{(m+n)\lambda}&\rightarrow v_{m\lambda}\otimes v_{n\lambda}
    \end{align*}
    is well-defined.
\end{proposition}
\begin{proof}
    The vector on the right-hand side is of $\mathfrak{b}$-highest weight $(m+n)\lambda$ and the tensor product is finite-dimensional hence also the $U(\mathfrak{g})$-module generated by $v_{m\lambda}\otimes v_{n\lambda}$. Thus, we know that this submodule is a quotient of $K_{\mathfrak{b}}((m+n)\lambda)$.
\end{proof}
We assume from now on that $\lambda\in\mathfrak{h}^{*}$ as well as $k\lambda\in\mathfrak{h}^{*}$ are dominant integral for any $k\in\N$. 

In turn, dualizing provides us with a map of $\mathfrak{g}$-modules
\begin{equation}\label{eq-mult-cartan}
    K_{\mathfrak{b}}(m\lambda)^*\otimes K_{\mathfrak{b}}(n\lambda)^*\rightarrow K_{\mathfrak{b}}((m+n)\lambda)^*
\end{equation}
We thus define the superalgebra
\begin{equation*}
    R(\lambda)=\bigoplus_{k\in\N_{0}}K_{\mathfrak{b}}(k\lambda)^*
\end{equation*}
with multiplication given by \eqref{eq-mult-cartan}.
\begin{remark}\label{rem-cartan-mult-well-defined}
    We specify what the multiplication from \eqref{eq-mult-cartan} looks like:

    We have the isomorphism of finite-dimensional $\mathfrak{g}$-modules, see \cite[Appendix A.2.3.]{Musson}
    \begin{align*}
        \gamma_{V,W}:W^*\otimes V^*&\rightarrow(V\otimes W)^*\\
        \varphi\otimes\psi&\mapsto\left(r\otimes s\mapsto\psi(r)\varphi(s)\right)
    \end{align*}
    Hence, the multiplication from \eqref{eq-mult-cartan} is the composition
    \begin{equation*}
        K_{\mathfrak{b}}(m\lambda)^*\otimes K_{\mathfrak{b}}(n\lambda)^*\xrightarrow{\gamma^{*}_{K_{\mathfrak{b}}(n\lambda),   K_{\mathfrak{b}}(m\lambda)}} ( K_{\mathfrak{b}}(n\lambda)\otimes  K_{\mathfrak{b}}(m\lambda))^*\xrightarrow{C_{n,m}^{*}}K_{\mathfrak{b}}((m+n)\lambda)^{*}
    \end{equation*}
    To see that this multiplication is indeed commutative, we note the following commutative diagram of $\mathfrak{g}$-modules
    \begin{equation*}
        % https://q.uiver.app/#q=WzAsNSxbMCwwLCJLX3tcXG1hdGhmcmFre2J9fShtXFxsYW1iZGEpXipcXG90aW1lcyBLX3tcXG1hdGhmcmFre2J9fShuXFxsYW1iZGEpXioiXSxbMiwwLCIoIEtfe1xcbWF0aGZyYWt7Yn19KG5cXGxhbWJkYSlcXG90aW1lcyAgS197XFxtYXRoZnJha3tifX0obVxcbGFtYmRhKSleKiJdLFs0LDEsIktfe1xcbWF0aGZyYWt7Yn19KChtK24pXFxsYW1iZGEpXnsqfSJdLFsyLDIsIiggS197XFxtYXRoZnJha3tifX0obVxcbGFtYmRhKVxcb3RpbWVzICBLX3tcXG1hdGhmcmFre2J9fShuXFxsYW1iZGEpKV4qIl0sWzAsMiwiS197XFxtYXRoZnJha3tifX0oblxcbGFtYmRhKV4qXFxvdGltZXMgS197XFxtYXRoZnJha3tifX0obVxcbGFtYmRhKV4qIl0sWzAsMSwiXFxnYW1tYV57Kn1fe0tfe1xcbWF0aGZyYWt7Yn19KG5cXGxhbWJkYSksICAgS197XFxtYXRoZnJha3tifX0obVxcbGFtYmRhKX0iXSxbMSwyLCJDX3tuLG19XnsqfSJdLFs0LDMsIlxcZ2FtbWFeeyp9X3tLX3tcXG1hdGhmcmFre2J9fShtXFxsYW1iZGEpLCAgIEtfe1xcbWF0aGZyYWt7Yn19KG5cXGxhbWJkYSl9IiwyXSxbMywyLCJDX3ttLG59XnsqfSIsMl0sWzQsMCwiXFx0YXVfe0tfe1xcbWF0aGZyYWt7Yn19KG5cXGxhbWJkYSleKixLX3tcXG1hdGhmcmFre2J9fShtXFxsYW1iZGEpXip9Il0sWzMsMSwiXFx0YXVfe0tfe1xcbWF0aGZyYWt7Yn19KG1cXGxhbWJkYSksS197XFxtYXRoZnJha3tifX0oblxcbGFtYmRhKX0iXV0=
        \begin{tikzcd}
        	{K_{\mathfrak{b}}(m\lambda)^*\otimes K_{\mathfrak{b}}(n\lambda)^*} && {( K_{\mathfrak{b}}(n\lambda)\otimes  K_{\mathfrak{b}}(m\lambda))^*} \\
        	&&&& {K_{\mathfrak{b}}((m+n)\lambda)^{*}} \\
        	{K_{\mathfrak{b}}(n\lambda)^*\otimes K_{\mathfrak{b}}(m\lambda)^*} && {( K_{\mathfrak{b}}(m\lambda)\otimes  K_{\mathfrak{b}}(n\lambda))^*}
        	\arrow["{\gamma^{*}_{K_{\mathfrak{b}}(n\lambda),   K_{\mathfrak{b}}(m\lambda)}}", from=1-1, to=1-3]
        	\arrow["{C_{n,m}^{*}}", from=1-3, to=2-5]
        	\arrow["{\tau_{K_{\mathfrak{b}}(n\lambda)^*,K_{\mathfrak{b}}(m\lambda)^*}}", from=3-1, to=1-1]
        	\arrow["{\gamma^{*}_{K_{\mathfrak{b}}(m\lambda),   K_{\mathfrak{b}}(n\lambda)}}"', from=3-1, to=3-3]
        	\arrow["{\tau_{K_{\mathfrak{b}}(m\lambda),K_{\mathfrak{b}}(n\lambda)}^{*}}", from=3-3, to=1-3]
        	\arrow["{C_{m,n}^{*}}"', from=3-3, to=2-5]
        \end{tikzcd},
    \end{equation*}
    where $\tau_{V,W}(v\otimes w)=(-1)^{|v|\cdot|w|}w\otimes v$ for $\mathfrak{g}$-modules $V$ and $W$.
    
    The triangle on the right-hand side commutes as $\tau_{K_{\mathfrak{b}}(m\lambda),K_{\mathfrak{b}}(n\lambda)}(v_{m\lambda}\otimes v_{n\lambda})=v_{n\lambda}\otimes v_{m\lambda}$, since $v_{m\lambda}$ and $v_{n\lambda}$ are assumed to be even.

    Further, the square on the left-hand-side commutes as it boils down to checking the equation
    \begin{equation*}
        (-1)^{|c|\cdot|d|}\varphi(d)\psi(c)=(-1)^{|\varphi|\cdot|\psi|}\varphi(d)\psi(c)
    \end{equation*}
    for $\mathfrak{g}$-modules $V,W$ with $\varphi\in V^*$, $\psi\in W^*$, $d\in V$ and $c\in W$, all assumed to be $\Z/2\Z$-homogeneous.

    The term $\varphi(d)\psi(c)$ can only be non-zero if $\varphi$ and $d$  as well as $\psi$ and $c$ share the same parity, respectively. Hence, the signs of both sides also agree.
\end{remark}
We recall the PBW theorem
\begin{theorem}[{\cite[Theorem 1.36]{Cheng-Wang}}]
    For a homogenous basis of $\mathfrak{g}$ with $x_1,\dots,x_n$ a basis of $\mathfrak{g}_{\overline{0}}$ and $y_1,\dots,y_q$ a basis of $\mathfrak{g}_{\overline{1}}$, the universal enveloping algebra $U(\mathfrak{g})$ has a basis of the form
    \begin{equation*}
        y_{i_{1}}\cdots y_{i_{l}}x_{1}^{k_{1}}\dots x_{n}^{k_{n}},\text{ for $k_i\geq 0$, $1\leq i_1<\cdots<i_l\leq q$}
    \end{equation*}
\end{theorem}

In particular, we have that $K_{\mathfrak{b}}(\lambda)=U(\mathfrak{n}^{-})v_\lambda$.
We fix a homogeneous basis of odd vectors $y_1,\dots,y_q$ and and even vectors $f_1,\dots,f_n$ of $\mathfrak{n}^{-}$.

Hence, we get that the module $K_{\mathfrak{b}}(\lambda)$ is spanned by
\begin{equation*}
    y_{i_{1}}\cdots y_{i_{l}}f_{1}^{k_{1}}\dots f_{n}^{k_{n}}v_{\lambda},\text{ for $k_i\geq 0$, $1\leq i_1<\cdots<i_l\leq q$}
\end{equation*}

In particular, we have the following filtration of $K_{\mathfrak{b}}(\lambda)$
\begin{equation*}
    K_{\mathfrak{b}}(\lambda)=\bigcup\limits_{k\in\N_0}U(\mathfrak{n}^{-})_kv_\lambda
\end{equation*}
where the space $U(\mathfrak{n}^{-})_kv_\lambda$ is spanned by
\begin{equation*}
    y_{i_{1}}\cdots y_{i_{l}}f_{1}^{k_{1}}\dots f_{n}^{k_{n}}v_{\lambda},\hspace{1em}l+k_1+\cdots+k_n\leq k
\end{equation*}\smallskip

We now consider monomial orders on the free supercommutative superalgebra $\mathbb{C}[x_1,\dots,x_n,\xi_1,\dots,\xi_q]$.
\begin{definition}[{\cite[Definition 3.3.3]{Gröbner_Super}}]
    Let $<$ be a total-order on $\N^n\times\{0,1\}^q$ such that $\alpha+\delta<\beta+\delta$ whenever $\alpha<\beta$, provided $\alpha+\delta,\beta+\delta\in\N^n\times\{0,1\}^q$,where $\alpha,\beta,\delta\in\N^n\times\{0,1\}^q$ and sums are taken componentwise in $\N^n\times\N^q$ (which contains $\N^n\times\{0, 1\}^q$).\\
    Such a total order on $\N^n\times\{0, 1\}^q$ induces a total order on the monomials of $\mathbb{C}[x_1,\dots,x_n,\xi_1,\dots,\xi_q]$, which will be called a \textit{quasi-monomial order} or simply \textit{(monomial) order}.\\
\end{definition}
\begin{remark}
    If we restrict a monomial order of $\N^n\times\N^q$ onto $\N^n\times\{0,1\}^q$, we call this restricted order \textit{induced}, following \cite{Gröbner_Super}.
\end{remark}
We are now going to assume we are working with an order $<$ on the monomials of $\mathbb{C}[x_1,\dots,x_n,\xi_1,\dots,\xi_q]$ that is induced from a monomial order of $\N^n\times\N^q$.

\begin{definition}
    Let $\dim\mathfrak{n}^{-}=n|q$ and fix $q$ integers $n+q\geq i_{q}>\cdots>i_{1}\geq 1$. Choose a basis of $\mathfrak{n}^{-}$ with elements $f_{n+q},\dots,f_{1}$ where $f_{i_{s}}$ is odd for all $1\leq s\leq q$. For non-negative integers $m_{n+q},\dots,m_{1}\in\Z_\geq0$ where $m_{i_{s}}\in\{0,1\}$, we define the ordered monomial for $\mathbf{m}=(m_{p+q},\dots,m_{1})$
    \begin{equation*}
        f^{\mathbf{m}}:=f_{n+q}^{m_{n+q}}\cdots f_{1}^{m_{1}}
    \end{equation*}

    If the ordering of the basis is fixed, then we are going to simplify notation by writing $(I,\mathbf{m})$, where $I=(m_{i_{q}},\dots,m_{i_{1}})$ and $\mathbf{m}$ is the ordered muli-exponent in the purely even elements.
\end{definition}

This leads us to the definition adapted from \cite{Favourable}. 

\begin{definition}
    We call a monomial $f^{(I,\mathbf{m})}$ \textit{essential} for the module $K_{\mathfrak{b}}(\lambda)$ with respect to the order $<$ if
    \begin{equation*}
    f^{(I,\mathbf{m})}v_\lambda\not\in\mathrm{span}\langle f^{(I',\mathbf{m}')}v_\lambda|(I',\mathbf{m}')<(I,\mathbf{m})\rangle
    \end{equation*}
    In that case, we refer to the monomial $(I,\mathbf{m})$ as \textit{essential} for $K_\mathfrak{b}(\lambda)$ as well.
\end{definition}

Before we go ahead and state the main property of being essential, we introduce
the following lemma.

\begin{lemma}\label{lemma-tensor-action-sign}
    Let $k,k'\in\N$ and let $v_{k\lambda}$ and $v_{k'\lambda}$ be the highest weight vectors of $K_{\mathfrak{b}}(k\lambda)$ and $K_{\mathfrak{b}}(k'\lambda)$.
    Then we have
    \begin{equation*}
        f^{(I,\mathbf{m})}.(v_{k\lambda}\otimes v_{k'\lambda})=\sum\limits_{(J',\mathbf{m}')+(J'',\mathbf{m}'')=(I,\mathbf{m})}(-1)^{K_{J',J''}}\prod\limits_{i=1}^{n}\begin{pmatrix}
            m_i\\
            m'_i
        \end{pmatrix}f^{(J',\mathbf{m}')}v_{k\lambda}\otimes f^{(J'',\mathbf{m}'')}v_{k'\lambda}
    \end{equation*}
    where $K_{J',J''}$ is the sum
    \begin{equation*}
        K_{J',J''}=\sum^q_{j=1}\sum^q_{i=j+1} J'_j\cdot J''_i
    \end{equation*}
\end{lemma}

\begin{proof}
    We are going to prove this via induction on the sum $|I|+|\mathbf{m}|$.
    There is nothing to show for monomials consisting of one even or one odd basis element.
    Write now
    \begin{equation*}
        f^{(I,\mathbf{m})}=f_{n+q}^{m_{n+q}}\cdots f_{1}^{m_{1}}
    \end{equation*}
    and let $1\leq t\leq n+q$ be maximal such that $m_{t}\neq0$.
    We are now going to distinguish between $f_{t}$ being odd or even.

    If $f_{t}$ is even, we first recall the formula for the action of $f_{t}^{m_{t}}$ on arbitrary elementary tensors being
    \begin{equation*}
        f_{t}^{m_{t}}(v\otimes w)=\sum_{r=0}^{m_{t}}\begin{pmatrix}
            m_{t}\\
            r
        \end{pmatrix}
        f^{r}v\otimes f^{m_{t}-r}w
    \end{equation*}
    Now, refer to $1\leq t_e\leq n$ as the index in the multi-index $\mathbf{m}=(m_1,\dots,m_n)$ indicating $m_t$ in $(I,\mathbf{m})$.
    Hence, we also get, where $e_{t}$ stands for the $t$-th standard vector,
    \begin{align*}
        f^{(I,\mathbf{m})}(v_{k\lambda}\otimes v_{k'\lambda})=&f_{t}^{m_{t}}(f^{(I,\mathbf{m}-m_{t}e_{t_{e}})}(v_{k\lambda}\otimes v_{k'\lambda}))\\
        =&f_{t}^{m_{t}}\left(\sum\limits_{(J',\mathbf{m}')+(J'',\mathbf{m}'')=(I,\mathbf{m}-m_{t}e_{t})}(-1)^{K_{J',J''}}\prod\limits_{i=1}^{t_{e}-1}\begin{pmatrix}
            m_i\\
            m'_i
        \end{pmatrix}f^{(J',\mathbf{m}')}v_{k\lambda}\otimes f^{(J'',\mathbf{m}'')}v_{k'\lambda}\right)\\
        =&\sum\limits_{\substack{(J',\mathbf{m}')+(J'',\mathbf{m}'')=(I,\mathbf{m}-m_{t}e_{t_{e}})\\r=0,\dots,m_{t_{e}}}}(-1)^{K_{J',J''}}\begin{pmatrix}
            m_t\\
            r
        \end{pmatrix}\\
        &\cdot\prod\limits_{i=1}^{t_{e}-1}\begin{pmatrix}
            m_i\\
            m'_i
        \end{pmatrix}f^{(J',\mathbf{m}'+re_{t_{e}})}v_{k\lambda}\otimes f^{(J'',\mathbf{m}''+(m_{t}-r)e_{t_{e}})}v_{k'\lambda}
    \end{align*}
    which, after reindexing, is exactly the result needed.

    If $f_{t}$ is odd, we can only have $m_t=1$. 
    Again, referring to $1\leq t_{e}\leq q$ as the index in the multi-index $I=(m_{i_{q}},\dots,m_{i_{1}})$ indicating $m_{t}$ in $(I,\mathbf{m})$ yields us
    \begin{align*}
        f^{(I,\mathbf{m})}(v_{k\lambda}\otimes v_{k'\lambda})=&f_{t}(f^{(I-e_{t_{e}},\mathbf{m})}(v_{k\lambda}\otimes v_{k'\lambda}))\\
        =&f_{t}\left(\sum\limits_{(J',\mathbf{m}')+(J'',\mathbf{m}'')=(I-e_{t_{e}},\mathbf{m})}(-1)^{K_{J',J''}}\prod\limits_{i=1}^{n}\begin{pmatrix}
            m_i\\
            m'_i
        \end{pmatrix}f^{(J',\mathbf{m}')}v_{k\lambda}\otimes f^{(J'',\mathbf{m}'')}v_{k'\lambda}\right)\\
        =&\sum\limits_{(J',\mathbf{m}')+(J'',\mathbf{m}'')=(I-e_{t_{e}},\mathbf{m})}(-1)^{K_{J',J''}}\prod\limits_{i=1}^{n}\begin{pmatrix}
            m_i\\
            m'_i
        \end{pmatrix}f^{(J'+e_{t_{e}},\mathbf{m}')}v_{k\lambda}\otimes f^{(J'',\mathbf{m}'')}v_{k'\lambda}\\
        +&\sum\limits_{(J',\mathbf{m}')+(J'',\mathbf{m}'')=(I-e_{t_{e}},\mathbf{m})}(-1)^{K_{J',J''}+|J'|}\prod\limits_{i=1}^{n}\begin{pmatrix}
            m_i\\
            m'_i
        \end{pmatrix}f^{(J',\mathbf{m}')}v_{k\lambda}\otimes f^{(J''+e_{t_{e}},\mathbf{m}'')}v_{k'\lambda}
    \end{align*}
    and now, we remark that if $J'+J''=I-e_{t_{e}}$, then
    \begin{equation*}
        K_{J',(J''+e_{t_{{e}}})}=\sum^q_{j=1}\sum^q_{i=j+1} J'_j\cdot (J''+e_{t_{e}})_i=\sum^q_{j=1}\sum^{t_{e}-1}_{i=j+1} J'_j\cdot J''_i+\sum^q_{j=1}J'_j=K_{J',J''}+|J'|
    \end{equation*}
    hence proving the claim.
\end{proof}
\begin{remark}
    As opposed to \cite[Proposition 2.11]{Favourable}, we need to be more explicit to be aware of the change of signs. This will be important when defining the monomial superalgebra in \eqref{eq-monomial-generators-ass-graded} into which we want to degenerate.
\end{remark}
\begin{definition}
    Let $(I,\mathbf{m}),\,(I',\mathbf{m}')\in\{0,1\}^q\times\N^n$. We say that they are \textit{compatible} if $(I,\mathbf{m})+(I',\mathbf{m}')\in\{0,1\}^q\times\N^n$.
 \end{definition}
 Being essential has the following property
\begin{proposition}\label{prop-essential-add}
    If $(I,\mathbf{m})$ and $(I',\mathbf{m}')$ are compatible and essential for $K_{\mathfrak{b}}(k\lambda)$ and $K_{\mathfrak{b}}(k'\lambda)$, respectively, then the monomial $(I+I',\mathbf{m}+\mathbf{m}')$ is essential for $K_{\mathfrak{b}}((k+k')\lambda)$
\end{proposition}
\begin{proof}
    We use the Cartan mapping from Proposition \ref{prop-cartan-embedding}. 

    Now, if we have an element $(U,\mathbf{r})<(I,\mathbf{m})$, then it can only be written as a sum $(U,\mathbf{r})=(U',\mathbf{r}')+(U'',\mathbf{r}'')$ with either $(U',\mathbf{r}')<(I,\mathbf{m})$ or $(U'',\mathbf{r}'')<(I',\mathbf{m}')$. But then the span of all of those $f^{(U,\mathbf{r})}(v_{k\lambda}\otimes v_{k'\lambda})$ does not contain the vector $f^{(I,\mathbf{m})}v_{k\lambda}\otimes f^{(I',\mathbf{m}')}v_{k'\lambda}$.
\end{proof}

\begin{definition}
    Denote the set of essential monomials for $K_\mathfrak{b}(k\lambda)$ by $\mathrm{es}(K_{\mathfrak{b}}(k\lambda),<)$. Then the set
    \begin{equation*}
        \Gamma(\lambda,<)=\bigcup\limits_{k\in\N_0}\mathrm{es}(K_{\mathfrak{b}}(k\lambda),<)\times\{k\}
    \end{equation*}
    forms together with an absorbing element $-\infty$ a semigroup where two monomials $(I,\mathbf{m})$ and $(I',\mathbf{m}')$ add up to $-\infty$ if their sum would not lie in $\{0,1\}^{q}\times\N^n$.
\end{definition}
\section{A basis and straightening laws for $R(\lambda)$}\label{sec-basis-straight-laws}
With the algebra $R(\lambda)$ defined, we now want to consider more closely how the structure constants with respect to the multiplication given by \eqref{eq-mult-cartan} behave.
\begin{definition}
    Consider the module $K_{\mathfrak{b}}(k\lambda)$ for $k\in\N$ and the basis of essential vectors $f^{(I,\mathbf{m})}v_{k\lambda}$ for $(I,\mathbf{m})\in\mathrm{es}(K_{\mathfrak{b}}(k\lambda),<)$. Then we refer to $\eta_{I,\mathbf{m},k}$ as the dual vector of $f^{(I,\mathbf{m})}v_{k\lambda}$ which then form a basis of $K_{\mathfrak{b}}(k\lambda)^{*}$.
\end{definition}
\begin{lemma}
    Let $(I',\mathbf{m}')$ be a multi-exponent and $(I,\mathbf{m})\in\mathrm{es}(K_{\mathfrak{b}}(k\lambda),<)$. Then we have $\eta_{I,\mathbf{m},k}(f^{(I',\mathbf{m}')}v_{k\lambda})=0$ if $(I',\mathbf{m}')<(I,\mathbf{m})$.
\end{lemma}
\begin{proof}
    If $(I',\mathbf{m}')$ is itself essential, then it is zero by the definition of the dual basis. Otherwise, i.e. it is not essential, $f^{(I',\mathbf{m}')}v_{k\lambda}$ lies in the span of $f^{(I'',\mathbf{m}'')}v_{k\lambda}$ with $(I'',\mathbf{m}'')<(I,\mathbf{m})$ of which we can again assume these $(I'',\mathbf{m}'')$ to be essential again. Applying to such a linear combination again $\eta_{I,\mathbf{m},k}$ yields us zero.
\end{proof}

Now, consider the multiplication of two dual basis vectors in $R(\lambda)$ corresponding to essential monomials
\begin{equation*}
    \eta_{I,\mathbf{m},k_{1}}\eta_{I',\mathbf{m}',k_{2}}=\sum\limits_{(I'',\mathbf{m}'')\in\mathrm{es}(K_{\mathfrak{b}}((k_{1}+k_{2})\lambda))}c^{(I'',\mathbf{m}'',k_{1}+k_{2})}_{(I,\mathbf{m},k_{1}),(I',\mathbf{m}',k_{2})}\eta_{I'',\mathbf{m}'',k_{1}+k_{2}}
\end{equation*}
with coefficients $c^{(I'',\mathbf{m}'',k_{1}+k_{2})}_{(I,\mathbf{m},k_{1}),(I',\mathbf{m}',k_{2})}\in\mathbb{C}$.

Then we have the following result
\begin{proposition}\label{prop-mult-straight}
    The coefficient $c^{(I'',\mathbf{m}'',k_{1}+k_{2})}_{(I,\mathbf{m},k_{1}),(I',\mathbf{m}',k_{2})}\in\mathbb{C}$ vanishes for $(I'',\mathbf{m}'')<(I,\mathbf{m})+(I',\mathbf{m}')$ and is non-zero for $(I'',\mathbf{m}'')=(I,\mathbf{m})+(I',\mathbf{m}')$
\end{proposition}
\begin{proof}
    We consider the formula from Lemma \ref{lemma-tensor-action-sign}. With Remark \ref{rem-cartan-mult-well-defined} in mind, we have for the coefficient in general
    \begin{align*}
        c^{(I'',\mathbf{m}'',k_{1}+k_{2})}_{(I,\mathbf{m},k_{1}),(I',\mathbf{m}',k_{2})}=&\sum\limits_{(J',\mathbf{n}')+(J'',\mathbf{n}'')=(I'',\mathbf{m}'')}(-1)^{K_{J',J''}}\prod\limits_{i=1}^{n}\begin{pmatrix}
            m''_i\\
            n'_i
        \end{pmatrix}\\
        \cdot&\eta_{I,\mathbf{m},k_{1}}(f^{(J'',\mathbf{n}'')}v_{k_{1}\lambda})\cdot\eta_{I',\mathbf{m}',k_{2}} (f^{(J',\mathbf{n}')}v_{k_{2}\lambda})
    \end{align*}
    Now note that for $(J',\mathbf{n}')+(J'',\mathbf{n}'')<(I,\mathbf{m})+(I',\mathbf{m}')$ either $(J'',\mathbf{n}'')<(I,\mathbf{m})$ or $(J',\mathbf{n}')<(I',\mathbf{m}')$ and hence that all summands would equal zero.

    In the case of $(I'',\mathbf{m}'')=(I,\mathbf{m})+(I',\mathbf{m}')$, we again need that $(J'',\mathbf{n}'')=(I,\mathbf{m})$ as well as $(J',\mathbf{n}')=(I',\mathbf{m}')$ as otherwise we would need that either $(J'',\mathbf{n}'')<(I,\mathbf{m})$ or $(J',\mathbf{n}')<(I',\mathbf{m}')$.

    Hence, the coefficient $c^{(I'',\mathbf{m}'',k_{1}+k_{2})}_{(I,\mathbf{m},k_{1}),(I',\mathbf{m}',k_{2})}$ is equal to
    \begin{equation*}
        (-1)^{K_{I',I}}\prod\limits_{i=1}^{n}\begin{pmatrix}
            m_i+m'_i\\
            m'_i
        \end{pmatrix}=(-1)^{K_{I',I}}\prod\limits_{i=1}^{n}\begin{pmatrix}
            m_i+m'_i\\
            m_i
        \end{pmatrix},
    \end{equation*}
    which is non-zero.
\end{proof}
\begin{remark}
Note that we have the equality
\begin{equation*}
    K_{I,J}+K_{J,I}=|I|\cdot|J|
\end{equation*}
which in particular implies that
\begin{equation*}
    (-1)^{K_{I,J}}=(-1)^{K_{J,I}}
\end{equation*}
if $|I|$ or $|J|$ is even and
\begin{equation*}
    (-1)^{K_{I,J}}=(-1)\cdot(-1)^{K_{J,I}}
\end{equation*}
if both $|I|$ and $|J|$ are odd.
\end{remark}
\begin{lemma}\label{lemma-structure-const-sign}
    The basis vectors can be refactored such that $c^{(I+I',\mathbf{m+}\mathbf{m}',k_{1}+k_{2})}_{(I,\mathbf{m},k_{1}),(I',\mathbf{m}',k_{2})}=(-1)^{K_{I',I}}$
\end{lemma}
\begin{proof}
    If we apply the renormalization
    \begin{equation*}
        f^{(I,\mathbf{m})}v_{k\lambda}\mapsto\prod\limits_{i=1}^{n+q}\frac{1}{m_i!}f^{(I,\mathbf{m})}v_{k\lambda}
    \end{equation*}
    (and thus also to $\eta_{I,\mathbf{m},k}$)
    then the binomial coefficients cancel, and we only end up with the coefficient $(-1)^{K_{I',I}}$.

    Note that for $i_{t}$ corresponding to an odd basis element, we have $m_{i_{t}}!=1$.
\end{proof}
\section{Favourable modules and and a vanishing ideal of $R(\lambda)$}\label{sec-favourable}
In this section, we want to introduce the notion of the highest module being favourable, similar to the notion in \cite{Favourable}. Further, we will retrieve an ideal by which the quotient is then isomorphic to $R(\lambda)$.
\begin{definition}
    We call the highest weight module $K_{\mathfrak{b}}(\lambda)$ \textit{favourable} if every essential monomial of $K_{\mathfrak{b}}(k\lambda)$ can be written as a sum of $k$ essential monomials of $K_{\mathfrak{b}}(\lambda)$.
\end{definition}

From now on, we assume $K_{\mathfrak{b}}(\lambda)$ to be favourable. 

We put a filtration on $R(\lambda)$. 
First, set for a multi-exponent $(I,\mathbf{m})\in\{0,1\}^{q}\times\N^n$ the subspaces
\begin{equation*}
    R_{k}(\lambda)^{>(I,\mathbf{m})}=\bigoplus_{\substack{(I',\mathbf{m}')\in\mathrm{es}(K_{\mathfrak{b}}(k\lambda),<)\\(I',\mathbf{m}')>(I,\mathbf{m})}}\langle\eta_{I',\mathbf{m}',k}\rangle_{\mathbb{C}}
\end{equation*}
as well as
\begin{equation*}
    R_{k}(\lambda)^{\geq(I,\mathbf{m})}=\bigoplus_{\substack{(I',\mathbf{m}')\in\mathrm{es}(K_{\mathfrak{b}}(k\lambda),<)\\(I',\mathbf{m}')\geq(I,\mathbf{m})}}\langle\eta_{I',\mathbf{m}',k}\rangle_{\mathbb{C}}
\end{equation*}
and then the subspaces of $R(\lambda)$
\begin{equation*}
    R(\lambda)^{>(I,\mathbf{m})}=\bigoplus_{k\in\N_0}\bigoplus_{\substack{(I',\mathbf{m}')\in\mathrm{es}(K_{\mathfrak{b}}(k\lambda),<)\\(I',\mathbf{m}')>(I,\mathbf{m})}}\langle\eta_{I',\mathbf{m}',k}\rangle_{\mathbb{C}}=\bigoplus_{k\in\N_0}R_{k}(\lambda)^{>(I,\mathbf{m})}
\end{equation*}
and 
\begin{equation*}
    R(\lambda)^{\geq(I,\mathbf{m})}=\bigoplus_{k\in\N_0}\bigoplus_{\substack{(I',\mathbf{m}')\in\mathrm{es}(K_{\mathfrak{b}}(k\lambda),<)\\(I',\mathbf{m}')\geq(I,\mathbf{m})}}\langle\eta_{I',\mathbf{m}',k}\rangle_{\mathbb{C}}=\bigoplus_{k\in\N_0}R_{k}(\lambda)^{\geq(I,\mathbf{m})}
\end{equation*}
Now, with Proposition \ref{prop-mult-straight}, we can see that both $R(\lambda)^{>(I,\mathbf{m})}$ and $R(\lambda)^{\geq(I,\mathbf{m})}$ are ideals of $R(\lambda)$.

Hence, with this filtration of ideals, we can consider the associated graded superalgebra
\begin{equation*}
    \mathrm{gr}(R(\lambda))=\bigoplus_{k\in\N_0}\bigoplus_{(I,\mathbf{m})\in\mathrm{es}(K_{\mathfrak{b}}(k\lambda),<)}\left(R_{k}(\lambda)^{\geq(I,\mathbf{m})}/R_{k}(\lambda)^{>(I,\mathbf{m})}\right)
\end{equation*}
We see that this superalgebra is generated by the monomials $\overline{\eta}_{I,\mathbf{m},1}$ for $(I,\mathbf{m})\in\mathrm{es}(K_{\mathfrak{b}}(\lambda),<)$ as we assume that $K_{\mathfrak{b}}(\lambda)$ is favourable.

Now, we want to identify the associated graded algebra with a subalgebra of $\mathbb{C}[v,x_1,\dots,x_n,\xi_1,\dots,\xi_q]$ generated by monomials.

However, we need to put in some care to account for the correct signs.

If we set for an essential monomial $((I,\mathbf{m}),k)=((i_1,\dots,i_q),\mathbf{m},k)\in\Gamma(\lambda,<)$, the monomial
\begin{equation}\label{eq-monomial-generators-ass-graded}
    \xi^{I}x^{\mathbf{m}}v^{k}:=\xi_{q}^{i_{q}}\cdots\xi_{1}^{i_{1}}x^{\mathbf{m}}v^{k},
\end{equation}
then one can verify that indeed
\begin{equation*}
    \mathrm{gr}(R(\lambda))\cong\mathbb{C}[\xi^{I}x^{\mathbf{m}}v|(I,\mathbf{m})\in\mathrm{es}(K_{\mathfrak{b}}(\lambda),<)],
\end{equation*}

Now, let $(I_1,\mathbf{m}_1),\dots,(I_t,\mathbf{m}_t)$ be the essential monomials in $\mathrm{es}(K_{\mathfrak{b}}(\lambda),<)$ with $|I_{i}|$ being even for $1\leq i\leq r$ and the remaining ones odd.
We go ahead and consider the surjection

\begin{equation}\label{eq-surjection-onto-assoc-graded}
    \begin{aligned}
        \mathbb{C}[x_1,\dots,x_r,\xi_1,\dots,\xi_s]&\rightarrow \mathrm{gr}(R(\lambda))\\
        x_i&\mapsto \overline{\eta}_{I_i,\mathbf{m}_i,1}\\
        \xi_j&\mapsto \overline{\eta}_{I_j+r,\mathbf{m}_j+r,1}
    \end{aligned}
\end{equation}

Thus, we can indeed put a grading of $\Gamma(\lambda,<)$ on $\mathbb{C}[x_1,\dots,x_r,\xi_1,\dots,\xi_s]$. 
We put the monomials sent to $0$ in the $-\infty$ component. 
This happens if and only if the sum of the multi-exponents would not lie in , which can also be seen from the identification in \eqref{eq-monomial-generators-ass-graded}.

Every other non-zero monomial $x_{i_{1}}\cdots x_{i_{q}}\xi_{j_{1}}\cdots\xi_{j_{p}}$ belongs to the component $(I,\mathbf{m},p+q)$ where
\begin{align*}
    I&=I_{i_{1}}+\cdots +I_{i_{q}}+I_{r+j_{1}}+\cdots+I_{r+j_{p}}\\
    \mathbf{m}&=\mathbf{m}_{i_{1}}+\cdots +\mathbf{m}_{i_{q}}+\mathbf{m}_{r+j_{1}}+\cdots+\mathbf{m}_{r+j_{p}}.
\end{align*}
By excluding that the monomial would be sent to $0$, we know that the sum is well-defined.

In particular, the map \eqref{eq-surjection-onto-assoc-graded} is a $\Gamma(\lambda,<)$-graded map. Hence, the ideal is generated by sums of monomials and individual monomials of the form
\begin{equation}\label{eq-relations-degenerate}
    x_{i_{1}}\cdots x_{i_{r}}\xi_{j_{1}}\cdots\xi_{j_{h-r}}+(-1)^{\sigma}x_{i'_{1}}\cdots x_{i'_{s}}\xi_{j'_{1}}\cdots\xi_{j'_{h-s}}, \hspace{1em}x_{i''_{1}}\cdots x_{i''_{t}}\xi_{j_{1}}\cdots\xi_{j''_{h-t}}
\end{equation}
with $\sigma\in\{0,1\}$. Note that both products in the binomial on the left-hand-side have the same total degree as they would need to map to the same monomial in $\mathrm{gr}(R(\lambda))$.

Now, we are interested in realizing relations for $R(\lambda)$.

For this, consider again generators of the ideal from \eqref{eq-relations-degenerate}, call them $\overline{g}_{1},\dots,\overline{g}_{m}$. First of all, we have that
\begin{equation*}
    \overline{g}_{k}(\overline{\eta}_{I_1,\mathbf{m}_1,1},\dots,\overline{\eta}_{I_t,\mathbf{m}_t,1})=0,\hspace{1em}\text{ for }1\leq k\leq m
\end{equation*}
Again, we have to distinguish between two different generator cases.

Firstly, suppose $\overline{g}_{k}$ is only a monomial. In that case, this can only happen if the sum of the multi-exponents would not be well-defined, i.e. to say there are two exponents $(I_1,\mathbf{m}_1)$ and $(I_2,\mathbf{m}_2)$ that are not compatble and their the sum $(I_1+I_2,\mathbf{m}_1+\mathbf{m}_2)$ would not lie in $\{0,1\}^{q}\times\N^n$. 
But, if we remember the structure constants from lemma \ref{lemma-structure-const-sign}, we see that this yields us in $R(\lambda)$
\begin{equation*}
    \overline{g}_{k}(\eta_{I_1,\mathbf{m}_1,1},\dots,\eta_{I_t,\mathbf{m}_t,1})=0
\end{equation*}

Now, we consider the case where $\overline{g}_{k}$ is a binomial and fix the multi-exponent being mapped to as $(J_{k},\mathbf{n}_{k},h_{k})$.

Then we know that $\overline{g}_{k}(\eta_{I_1,\mathbf{m}_{1},1},\dots,\eta_{I_t,\mathbf{m}_{t},1})\in R_{h_{k}}(K_{\mathfrak{b}}(\lambda))^{>(J_{k},\mathbf{n}_{k})}$, i.e. via Proposition \ref{prop-mult-straight}
\begin{equation*}
    \overline{g}_{k}(\eta_{I_1,\mathbf{m}_1,1},\dots,\eta_{I_t,\mathbf{m}_t,1})=\sum_{\substack{(J',\mathbf{n}')\in\mathrm{es}(K_{\mathfrak{b}}(h_{k}\lambda),<)\\(J',\mathbf{n}')>(J,\mathbf{n})}}a_{(J',\mathbf{n}')}\eta_{J',\mathbf{n}',h_{k}}
\end{equation*}
for some coefficients $a_{(J',n')}\in\C$.

Now, we know that again from proposition \ref{prop-essential-add} that $\eta_{J',n',h_{k}}$ can be expressed as a product of total degree $h_{k}$ in the essential monomials $(I_1,\mathbf{m}_1),\dots,(I_t,\mathbf{m}_t)$ plus a linear combination of vectors $\eta_{J'',\mathbf{n}'',h_{k}}$ with $(J'',\mathbf{n}'')>(J',\mathbf{n}')$ and $(J'',\mathbf{n}'')\in\mathrm{es}(K_{\mathfrak{b}}(h\lambda),<)$. 
However, since $K_{\mathfrak{b}}(h_{k}\lambda)$ is finite-dimensional, we will end up with a finite sum and hence an element of $\C[x_1,\dots,x_r,\xi_1,\dots,\xi_s]$ being
\begin{equation*}
    g_{k}=\overline{g}_{k}+\sum\limits_{j=1}^{r_{k}}g_{k,j}
\end{equation*}
where $g_{k,j}$ are polynomials in $\C[x_1,\dots,x_n,\xi_1,\dots,\xi_q]$ of which all monomials belong to the same $\Gamma(\lambda,<)$-component of some $(U_{k,j},\mathbf{v}_{k,j},h_{k})$ where $(U_{k,j},\mathbf{v}_{k,j})>(J_{k},\mathbf{n}_{k})$ as well as
\begin{equation*}
    g_{k}(\eta_{I_1,\mathbf{m}_1,1},\dots,\eta_{I_t,\mathbf{m}_t,1})=0
\end{equation*}

Consider now the surjective homomorphism of superalgebras
\begin{equation}\label{eq-surjection-onto-flag}
    \begin{aligned}
        \Phi:\mathbb{C}[x_1,\dots,x_r,\xi_1,\dots,\xi_s]/\langle g_1,\dots,g_t\rangle&\rightarrow R(\lambda),\\
        x_i&\mapsto \eta_{I_i,\mathbf{m}_i,1},\\
        \xi_j&\mapsto \eta_{I_j+r,\mathbf{m}_j+r,1}.
    \end{aligned}
\end{equation}
\begin{proposition}\label{prop-iso-flag}
    The map $\Phi$ from \eqref{eq-surjection-onto-flag} is an isomorphism.
\end{proposition}
\begin{proof}
    We are going to put a decreasing filtration compatible with $\Phi$ onto the space of the left-hand-side of \eqref{eq-surjection-onto-flag} and show that the associated graded map $\mathrm{gr}\,\Phi$ is injective.

    To simplify notation, we define $\mathcal{S}:=\mathbb{C}[x_1,\dots,x_r,\xi_1,\dots,\xi_s]$.

    For a multi-exponent $(I,\mathbf{m})\in\{0,1\}^{q}\times\N^n$, we define $\mathcal{S}_{n}^{\geq(I,\mathbf{m})}$ as the span of all monomials in $x_{i_{1}}\cdots x_{i_{q}}\xi_{j_{1}}\cdots\xi_{j_{p}}$ of total degree $n$ and $(K,\mathbf{f})\geq(I,\mathbf{m})$
    where
    \begin{align*}
        K&=I_{i_{1}}+\cdots +I_{i_{q}}+I_{r+j_{1}}+\cdots+I_{r+j_{p}}\\
        \mathbf{f}&=\mathbf{m}_{i_{1}}+\cdots +\mathbf{m}_{i_{q}}+\mathbf{m}_{r+j_{1}}+\cdots+\mathbf{m}_{r+j_{p}},
    \end{align*}
    if the sum lies in $\{0,1\}^{q}\times\N^n$.
    Then the space
    \begin{equation*}
        \mathcal{S}^{\geq(I,\mathbf{m})}=\bigoplus_{k\in\N_{0}}\mathcal{S}_{k}^{\geq(I,\mathbf{m})}
    \end{equation*}
    is an ideal.

    For $-\infty$ we define the space $\mathcal{S}_{n}^{\geq-\infty}$ as the span of all monomials $x_{i_{1}}\cdots x_{i_{q}}\xi_{j_{1}}\cdots\xi_{j_{p}}$ of total degree $n$ whose sum of respective essential monomials would sum up to $-\infty$ in $\Gamma(\lambda,<)$.

    Of course, we have a projection from $\mathcal{S}$ onto $A:=\mathbb{C}[x_1,\dots,x_r,\xi_1,\dots,\xi_s]/\langle g_1,\dots,g_t\rangle$, call it $p$, which means that we also have a filtration on $A$ via $p(\mathcal{S}^{\geq(I,\mathbf{m})})$. 
    Note that the monomials in the $-\infty$-component are zero in $A$.

    We now consider the following commutative diagram
    \begin{equation*}
        % https://q.uiver.app/#q=WzAsMyxbMCwwLCJTL1xcbGFuZ2xlXFxvdmVybGluZXtnfV97a31cXHJhbmdsZSJdLFsyLDAsIlxcbWF0aHJte2dyfUEiXSxbNCwwLCJcXG1hdGhybXtncn1SKE0pIl0sWzAsMSwiIiwwLHsic3R5bGUiOnsiaGVhZCI6eyJuYW1lIjoiZXBpIn19fV0sWzEsMiwiXFxtYXRocm17Z3J9XFxQaGkiXSxbMCwyLCJcXGNvbmciLDIseyJjdXJ2ZSI6NH1dXQ==
        \begin{tikzcd}
        {\mathcal{S}/\langle\overline{g}_{k}\rangle} && {\mathrm{gr}A} && {\mathrm{gr}R(M)}
        \arrow[two heads, from=1-1, to=1-3]
        \arrow["\cong"', curve={height=24pt}, from=1-1, to=1-5]
        \arrow["{\mathrm{gr}\Phi}", from=1-3, to=1-5]
        \end{tikzcd}
    \end{equation*}

    The first surjection stems from the fact that $\overline{g}_{k}$ and $g_{k}$ coincide in $\mathrm{gr}A$. And the isomorphism is based on the definition of the polynomials $\overline{g}_{k}$.

    Based on the diagram, one can then conclude that the map $\mathrm{gr}\,\Phi$, and hence $\Phi$ itself, has to be injective.
\end{proof}
\section{Flat degeneration}\label{sec-flat-family}
Now, we can go ahead and construct the degeneration. We are going to proceed as follows: First of all, we are constructing the one-parameter family in section \ref{subsec-construct-family}. Then, in section \ref{subsec-free-basis}, we are going to verify that this superalgebra is free as a $\C[t]$-module. We are thus taking a more algebraic approach as opposed to the geometric argument from \cite{Favourable} that a surjective morphism from an irreducible variety onto $\mathbb{A}^1$ is flat. 
While it is possible to replicate this result using the definition of reducedness of Jankowski from \cite[Definition 2.4]{Jankowski2025}, it would be rather difficult to verify if $\mathcal{R}$ is reduced.

\subsection{The ideal of the family}\label{subsec-construct-family}

If we now consider the generators of the quotient
\begin{equation}\label{eq-gen-flag}
    g_{k}=\overline{g}_{k}+\sum\limits_{j=1}^{r_{k}}g_{k,j}
\end{equation}
where $g_{k,j}$ are polynomials in $\C[x_1,\dots,x_n,\xi_1,\dots,\xi_q]$ of which all monomials belong to the same $\Gamma(\lambda,<)$-component of some $(U_{k,j},\mathbf{v}_{k,j},h_{k})$ where $(U_{k,j},\mathbf{v}_{k,j})>(J_{k},\mathbf{n}_{k})$ and $\overline{g}_{k}$ being of degree $(J_{k},\mathbf{n}_{k},h_{k})$.

Since the set of generators is finite and the monomial order is induced, we can apply the result of classical Gröbner theory, see \cite[Lemma 1.2.11]{Approx-Weight}, providing us with a vector $w\in\Z^{q|n}$ such that $w(U_{k,j},\mathbf{v}_{k,j})>w(J_{k},\mathbf{n}_{k})$ for all $1\leq k\leq m$ and $1\leq j\leq r_{k}$.

Before we go ahead and define our $\mathbb{C}[t]$-superalgebra, we introduce some notation to define the ideal more easily.

\begin{definition}\label{definition-max-exchange}
    Let $h\in\N_0$ and let $(I^{\mathrm{h,max}},\mathbf{m}^{\mathrm{h,max}})$ be the maximal element in $\mathrm{es}(K_\mathfrak{b}(h\lambda),<)$. 
    Then we define the following subset of the generators of \eqref{eq-relations-degenerate}
    \begin{align*}
        \mathcal{I}'_h:=\{&x_{i_{1}}\cdots x_{i_{r}}\xi_{j_{1}}\cdots\xi_{j_{h-r}}+(-1)^{\sigma}x_{i'_{1}}\cdots x_{i'_{s}}\xi_{j'_{1}}\cdots\xi_{j'_{h-s}}\,|\,\\
        &I^{\mathrm{h,max}}=I_{i_{1}}+\cdots +I_{i_{r}}+I_{r+j_{1}}+\cdots+I_{r+j_{h-r}}\\
        &\mathbf{m}^{\mathrm{h,max}}=\mathbf{m}_{i_{1}}+\cdots +\mathbf{m}_{i_{r}}+\mathbf{m}_{r+j_{1}}+\cdots+\mathbf{m}_{r+j_{h-r}}\}.
    \end{align*}
    Intuitively speaking, we collect all binomial relations in $\mathrm{gr}\,R(\lambda)$ from \eqref{eq-relations-degenerate} in the $\Gamma(\lambda,<)$-component $(I^{\mathrm{h,max}},\mathbf{m}^{\mathrm{h,max}},h)$.
\end{definition}

We now define a new superalgebra
\begin{equation*}
    \mathcal{R}=\mathcal{S}[t]/\mathcal{I}
\end{equation*}
with $\mathcal{I}$ being the ideal generated by the following set
\begin{align*}
    \left\{g_{k}=\overline{g}_{k}+\sum\limits_{j=1}^{r_{k}}t^{w(U_{k,j},\mathbf{v}_{k,j})-w(J_{k},\mathbf{n}_{k})}g_{k,j}\,\middle|\, 1\leq k\leq m\right\}\cup\bigcup_{h\in\N_0}\mathcal{I}'_h
\end{align*}
\subsection{Constructing a $\mathbb{C}[t]$-basis}\label{subsec-free-basis}
We aim to show that the superalgebra $\mathcal{R}$ is indeed free as a $\mathbb{C}[t]$-module. To achieve this, we are going to show that the basis of the algebra $R(\lambda)$ multiplied with powers of $t$ forms a $\C$-basis of $\mathcal{R}$.

We begin by modifying the filtration from Proposition \ref{prop-iso-flag}. 
For a multi-exponent $(I,\mathbf{m})\in\{0,1\}^{q}\times\N_0^n$ and a natural number $k\in\N_0$, we define $\mathcal{R}_{n}^{\geq((I,\mathbf{m}),k)}$ as the span of all monomials $\overline{x_{i_{1}}\cdots x_{i_{q}}\xi_{j_{1}}\cdots\xi_{j_{p}}t^{r}}$ of total degree $n=p+q$ in $x$ and $\xi$ and $(K,\mathbf{f})>(I,\mathbf{m})$ where
    \begin{align*}
        K&=I_{i_{1}}+\cdots +I_{i_{q}}+I_{r+j_{1}}+\cdots+I_{r+j_{p}}\\
        \mathbf{f}&=\mathbf{m}_{i_{1}}+\cdots +\mathbf{m}_{i_{q}}+\mathbf{m}_{r+j_{1}}+\cdots+\mathbf{m}_{r+j_{p}},
    \end{align*}
and $r$ arbitrary or $(K,\mathbf{f})=(I,\mathbf{m})$ and $r\geq k$. And then, we define
\begin{equation*}
    \mathcal{R}^{\geq((I,\mathbf{m}),k)}=\sum_{n\in\N_0}\mathcal{R}_{n}^{\geq((I,\mathbf{m}),k)}
\end{equation*}

Now, we consider the associated graded space
\begin{equation*}
    \mathrm{gr}\,\mathcal{R}=\bigoplus_{(I,\mathbf{m})\in\{0,1\}^{q}\times\N_{0}^n}\bigoplus_{k\in\N_{0}}\left(\mathcal{R}^{\geq((I,\mathbf{m}),k)}/\mathcal{R}^{\geq((I,\mathbf{m}),k+1)}\right)
\end{equation*}

In particular, we have a projection
\begin{align*}
    \pi:\mathcal{S}[t]/\langle \overline{g}_{k}\rangle&\twoheadrightarrow\mathrm{gr}\,\mathcal{R}\\
    \overline{x_i}&\mapsto (({I_i,\mathbf{m}_i}),0,\overline{x_i})\\
    \overline{\xi}_j&\mapsto (({I_j,\mathbf{m}_j}),0,\overline{\xi_j})\\
    t&\mapsto((0,0),1,\overline{t})
\end{align*}

Via the isomorphism $\mathcal{S}[t]\langle\overline{g}_{k}\rangle\cong\mathrm{gr}\,R(\lambda)\otimes_\C\C[t]$ we have a spanning set given in $\mathrm{gr}\,\mathcal
{R}$ by the images of $\eta_{I,\mathbf{m},h}\cdot t^{r}$ for $(I,\mathbf{m})\in\mathrm{es}(K_{\mathfrak{b}}(h\lambda),<)$ and $r\in\N_0$ arbitrary.

We claim now
\begin{lemma}
    The images $\pi(\eta_{I,\mathbf{m},h}\cdot t^{r})=:\theta_{I,\mathbf{m},h}\cdot t^{r}$ are $\C$-linearly independent in $\mathrm
    {gr}\,\mathcal{R}$.
\end{lemma}
\begin{proof}
    It suffices to only consider one subquotient $\left(\mathcal{R}^{\geq((I,\mathbf{m}),k)}/\mathcal{R}^{\geq((I,\mathbf{m}),k+1)}\right)$ for a fixed $(I,\mathbf{m})$ and $k$.

    Hence, we consider now coefficients $\lambda_{I,\mathbf{m},k,h}\in\C$ such that
    \begin{equation}\label{eq-power-basis-indep}
        \sum_{h\in\N_0}\lambda_{I,\mathbf{m},k,h}\cdot\theta_{I,\mathbf{m},h}\cdot t^{k}\in\mathcal{R}^{\geq((I,\mathbf{m}),k+1)}
    \end{equation}
    Before we go further, we need to consider how we can rewrite monomials that would sum to $(I,\mathbf{m})$.

    Suppose now that $x_{i_{1}}\cdots x_{i_{q}}\xi_{j_{1}}\cdots\xi_{j_{p}}\in\mathcal{S}$ is a monomial of total degree $q+p$ and the corresponding sum in $\{0,1\}^{q}\times\N^{n}$ equals $(I,\mathbf{m})$.

    Then we can write in $\mathcal{S}$
    \begin{equation*}
        x_{i_{1}}\cdots x_{i_{q}}\xi_{j_{1}}\cdots\xi_{j_{p}}+(-1)^{\sigma}\xi^{I}x^{\mathbf{m}}=\sum\limits_{k=1}^{m}f_{k}\overline{g}_{k}
    \end{equation*}
    for some $\sigma\in\{0,1\}$.
    
    Now we remember that the ring $\mathcal{S}$ is graded via $\Gamma(\lambda,<)$ which means that each $f_{k}$ can be assumed to be homogenous of degree $((I-J_{k},\mathbf{m}-\mathbf{n}_{k}),p+q-h_{k})$ if the difference can be defined. If not, then $f_{k}=0$.

    However, then we also get
    \begin{align*}
        &\sum\limits_{k=1}^{m}f_{k}\overline{g}_{k}+\sum\limits_{j=1}^{r_{k}}f_{k}t^{w(U_{k,j},\mathbf{v}_{k,j})-w(J_{k},\mathbf{n}_{k})}g_{k,j}=\\
        &x_{i_{1}}\cdots x_{i_{q}}\xi_{j_{1}}\cdots\xi_{j_{p}}+(-1)^{\sigma}\xi^{I}x^{\mathbf{m}}+\sum_{k=1}^{m}\sum\limits_{j=1}^{r_{k}}f_{k}t^{w(U_{k,j},\mathbf{v}_{k,j})-w(J_{k},\mathbf{n}_{k})}g_{k,j}\in\mathcal{I}
    \end{align*}

    But then, in the subquotient $\left(\mathcal{R}^{\geq((I,\mathbf{m}),k)}/\mathcal{R}^{\geq((I,\mathbf{m}),k+1)}\right)$, we would get that
    \begin{equation*}
         \overline{x_{i_{1}}\cdots x_{i_{q}}\xi_{j_{1}}\cdots\xi_{j_{p}}}=-(-1)^{\sigma}\overline{\xi^{I}x^{\mathbf{m}}}
    \end{equation*}

    This implies that we can restate \eqref{eq-power-basis-indep} to be
    \begin{align*}
        &\sum_{k'\geq k}\sum_{h\in\N_0}\lambda_{I,\mathbf{m},k',h}\cdot\theta_{I,\mathbf{m},h}\cdot t^{k'}\in\\
        &\mathrm{span}\langle \overline{x_{i_{1}}\cdots x_{i_{q}}\xi_{j_{1}}\cdots\xi_{j_{p}}t^r}|
        (I_{i_{1}},{\mathbf{m}_{i_{1}}})+\cdots+(I_{i_{q}},{\mathbf{m}_{i_{q}}})+(I_{j_{1}},{\mathbf{m}_{j_{1}}})+\cdots+(I_{j_{p}},{\mathbf{m}_{j_{p}}})>(I,\mathbf{m})\rangle
    \end{align*}
    Now we specialise $t$ to an arbitrary $a\in\C\setminus\{0\}$. After applying the isomorphism
    \begin{align*}
        \mathcal{R}/\langle t-a\rangle&\rightarrow R(\lambda)\\
        x_i&\mapsto a^{-w(I_{i},\mathbf{m}_i)}x_i\\
        \xi_j&\mapsto a^{-w(I_{j+r},\mathbf{m}_{j+r})}\xi_j\\
    \end{align*}
    we get that
    \begin{equation*}
        \sum_{h\in\N_0}\eta_{I,\mathbf{m},h}\left(\sum_{k'\geq k}\lambda_{I,\mathbf{m},k',h}\cdot a^{k'-w(I,\mathbf{m})}\right)\in R(\lambda)^{>(I,\mathbf{m})}
    \end{equation*}

    This implies, in particular, that
    \begin{equation*}
        \left(\sum_{k'\geq k}\lambda_{I,\mathbf{m},k',h}\cdot a^{k'-w(I,\mathbf{m})}\right)=0
    \end{equation*}
    However, we can let $a$ be any non-zero complex number; hence the corresponding polynomial has infinitely many roots, implying that $\lambda_{I,\mathbf{m},k',h}=0$ for any $k'\geq k$ and $h\in\N_0$ yielding the claim.
\end{proof}

This claim tells us that the vectors $\theta_{I,\mathbf{m},h}\cdot t^{r}$ are linearly independent in $\mathcal{R}$.
We now want to verify that they generate $\mathcal{R}$.

Hence, we now prove the next lemma.
\begin{lemma}\label{lemma-generating-standards}
    The vectors $\theta_{I,\mathbf{m},h}\cdot t^{r}$ form a $\mathbb{C}$-generating set as a vector space of $\mathcal{R}$. 
\end{lemma}
\begin{proof}
    It suffices to show that the image of any monomial $x_{i_{1}}\cdots x_{i_{q}}\xi_{j_{1}}\cdots\xi_{j_{p}}$ can be generated. Fix the graded component to be $(J,\mathbf{n},p+q)$
    We can go by induction on the monomial order of the essential monomials in $K_{\mathfrak{b}}((p+q)\lambda)$.

    If $(J,\mathbf{n})$ is already the biggest essential monomial, then we already have by Definition \ref{definition-max-exchange}
    \begin{equation*}
        x_{i_{1}}\cdots x_{i_{r}}\xi_{j_{1}}\cdots\xi_{j_{h-r}}=-(-1)^{\sigma}x_{i'_{1}}\cdots x_{i'_{s}}\xi_{j'_{1}}\cdots\xi_{j'_{h-s}}
    \end{equation*}
    in $\mathcal{R}$.
    And if $(J,\mathbf{n})$ is now smaller, then we have similar to the previous proof
    \begin{equation*}
        x_{i_{1}}\cdots x_{i_{q}}\xi_{j_{1}}\cdots\xi_{j_{p}}=-(-1)^{\sigma}\xi^{I}x^{m}+\sum\limits_{j=1}^{r_{k}}f_{k}t^{w(U_{k,j},\mathbf{v}_{k,j})-w(J_{k},\mathbf{n}_{k})}g_{k,j}
    \end{equation*}
    And since the terms $f_{k}t^{w(U_{k,j},\mathbf{v}_{k,j})-w(J_{k},\mathbf{n}_{k})}g_{k,j}$ all are spanned by monomials greater than $(J,\mathbf{n})$ multiplied with some power of $t$, we can apply the induction hypothesis. Hence, the claim follows.
\end{proof}
\begin{corollary}
    The vectors $\theta_{I,\mathbf{m},h}\cdot t^{r}$ for $(I,\mathbf{m})\in\mathrm{es}(K_{\mathfrak{b}}(h\lambda),<)$ and $r\in\N_0$ form a $\mathbb{C}$-basis of $\mathcal{R}$. Hence the vectors $\theta_{I,\mathbf{m},h}$ for $(I,\mathbf{m})\in\mathrm{es}(K_{\mathfrak{b}}(h\lambda),<)$ and $h\in\N_0$ form a $\C[t]$-basis of $\mathcal{R}$ implying that this $\C[t]$-module is free and thus, in particular, flat.
\end{corollary}
\section{The supergeometric interpretation}\label{sec-supergeom}
In this section, we are going to interpret the previous result in a geometric context, mostly following \cite{Fioresi}. Another reference the author recommends are the papers by Sherman like \cite{Sherman}.

We are going to state a few fundamentals from supergeometry and also answer the question when the monomial superalgebra $\mathrm{gr}\,R(\lambda)$ is a toric supervariety \cite{Jankowski2024, Jankowski2025}.\smallskip

Let $A=A_{\overline{0}}\oplus A_{\overline
1}$ be a commutative superalgebra. 
Then $A$ itself forms an $A_{\overline{0}}$-module and we can define a sheaf of $\mathcal{O}_{\mathrm
{Spec}\,A_{\overline
0}}$-modules, call it $\mathcal
O_{A}$, such that $\mathcal{O}_{A}(D(f))=A_{f}$ for any $f\in A_{\overline{0}}$. 
In fact, it is a sheaf of commutative superalgebras. This space with the sheaf $(|\mathrm{Spec}\,A_{\overline{0}}|,\mathcal
{O}_{A})$ is then referred to as an \textit{affine superscheme}. The category of affine superschemes is anti-equivalent to the category of commutative superalgebras.

A general \textit{superscheme} is a topological space with a sheaf of commutative superalgebras with local stalks with an open covering by affine superschemes. \cite[Proposition 10.1.3]{Fioresi}\smallskip

Similar to the classical case, one can define morphisms for superschemes as well as fibres \cite[p. 5407]{Zubkov-Dimension-Theory} and flatness. \cite[Section 3]{Masuoka-Group-Superschemes}
This leads us then to the result.
\begin{corollary}\label{result1}
    The morphism of superschemes $\kappa:\mathrm{Spec}\,\mathcal{R}\rightarrow\mathbb{A}^{1|0}$ is flat and the fibers are
    \begin{align*}
        \kappa^{-1}(a)\cong\begin{cases}
            \mathrm{Spec}\,R(\lambda),\hspace{1em}a\neq0\\
            \mathrm{Spec}\,\mathbb{C}[\xi^{I}x^{\mathbf{m}}v|(I,\mathbf{m})\in\mathrm{es}(K_{\mathfrak{b}}(\lambda),<)],\hspace{1em}a=0
        \end{cases}
    \end{align*}
\end{corollary}

Now, we are going to verify that the superalgebras one gets from the constructions from \cite{Kus-Fourier} are indeed toric supervarieties according to \cite{Jankowski2025}.

First of all, we recall the notions of supertori, mostly following \cite{Jankowski2024} and \cite{Jankowski2025}.
\begin{definition}[{\cites[Definition 4.1]{Jankowski2024}[Definition 4.1]{Jankowski2025}}]
    An algebraic supertorus is an algebraic supergroup $T = (T_0 , \mathfrak{t})$ given by a Harish-Chandra pair, \cite[p. 164]{Serg_Rep}, such that $T_0\cong (\C^\times)^n$ is an ordinary algebraic torus and $[\mathfrak{t}_0,\mathfrak{t}]=0$.
\end{definition}
Now note that the corresponding superalgebra is of the form $\C[t_{1}^{\pm},\dots,t_{p}^{\pm},\xi_1,\dots,\xi_q]$.

We also get a $\mathfrak{t}$-action on $\C[t_{1}^{\pm},\dots,t_{p}^{\pm},\xi_1,\dots,\xi_q]$. 
Namely, if $\mathfrak{t}=\C\{x_1,...,x_n\}\oplus\C\{\theta_1,...,\theta_s\}$ and $[\theta_i,\theta_j] = (x_{ij})_1x_1+\dots + (x_{ij})_px_p\in\mathfrak{t}_{\overline{0}}$ for $1\leq i,j\leq q$, then $\mathfrak{t}$ acts von $\C[T]$ via
\begin{align}\label{eq-deriv-odd-toric}
    x_i&\mapsto-t_i\frac{\partial}{\partial t_i}\\
    \theta_i&\mapsto\sum_{j=1}^{q}\xi_j\left\langle \frac{t\partial}{\partial t},x_{ij}\right\rangle-\frac{\partial}{\partial\xi_i},
\end{align}
where $\left\langle \frac{t\partial}{\partial t},x_{ij}\right\rangle=(x_{ij})_1t_1\frac{t\partial}{\partial t_1}+\dots + (x_{ij})_pt_p\frac{t\partial}{\partial t_p}$
Before we can state the notion of supervariety, we also need some algebraic definitions from \cite{Jankowski2025} for arbitrary superrings \cite[Definition 2.1]{Jankowski2025}.
\begin{definition}[{\cite[Definition 2.2]{Jankowski2025}}]
    Let $A$ be a superring. The \textit{total superring of fractions} of $A$ is $K(A):=\mathrm{nonzdiv}(A)^{-1}A$, where $\mathrm{nonzdiv}(A)\subseteq A_{\overline{0}}$ is the set of non-zero divisor. We denote by $i:A \rightarrow K(A)$ the natural inclusion, by $J_A$ the ideal $i^{-1}((K(A)_{\overline{1}})) \subseteq A$, and by $\overline{A}$ the \textit{underlying ring} $A/J_A$. If $S$ is any subset of $A$, we write $\overline{S}$ for its image in $\overline{A}$.
\end{definition}
\begin{definition}[{\cite[Definition 2.4]{Jankowski2025}}]
    Let $A$ be a superring.
    \begin{enumerate}[label=(\alph*)]
        \item $A$ is a \textit{superfield} if $J_A = A\setminus A^\times$
        \item $A$ is an \textit{integral superdomain} if $J_A = \mathrm{zdiv}(A)$, where $\mathrm{zdiv}(A)$ is the set of zero-divisors in $A$
        \item $A$ is \textit{reduced} if $J_A = \mathrm{Nil}(A)$ and $\overline{\mathrm{zdiv}(A)} = \mathrm{zdiv}(\overline{A})$
    \end{enumerate}
\end{definition}

\begin{definition}[{\cite[Definition 2.10]{Jankowski2025}}]
Let $A$ be a superring.
    \begin{enumerate}[label=(\alph*)]
        \item $A$ is \textit{fermionically regular} (or FR) if it is of the form $\bigwedge M$ for some commutative ring $R$ and $R$-module $M$. In this situation, we have $J_A = (A_{\overline{1}})$, and we may assume $R\cong \overline{A}$ and $M\cong J_A/J_A^2$.
        \item $A$ is \textit{generically fermionically regular} (or GFR) if $A_\mathfrak{p}$ is a FR superring for all minimal prime ideals $\mathfrak{p}$.
        \item $A$ is GFRR if it is GFR and reduced (so that each $A_\mathfrak{p}$ is a FR superfield for $\mathfrak{p}$ minimal).
    \end{enumerate}
\end{definition}
Jankowski also defined similar notions for superschemes.
\begin{definition}[{\cite[Definition 3.2]{Jankowski2025}}]
    Let $X$ be a superscheme.
    \begin{enumerate}[label=(\alph*)]
        \item $X$ is \textit{irreducible} if $|X|$ is irreducible as a topological space.
        \item $X$ is \textit{reduced} if $\mathscr{O}_X(U)$ is a reduced superring for each affine open $U$.
        \item $X$ is \textit{reduced} if $\mathscr{O}_X(U)$ is a reduced superring for each affine open $U$.
        \item $X$ is \textit{integral} if $\mathscr{O}_X(U)$ is an integral superdomain for each affine open $U$.
        \item $X$ is \textit{GFR} if there is a collection of points $\{\eta\}$ whose closure is $X$ and for which
each $\mathscr{O}_{X,\eta}$ is FR.
\item $X$ is \textit{GFRR} if it is GFR and reduced.
    \end{enumerate}
\end{definition}
\begin{definition}[{\cite[Definition 3.5]{Jankowski2025}}]
    An algebraic supervariety is a GFRR separated superscheme of finite type over an algebraically closed field.
\end{definition}
This then allows us to state the definition of a toric supervariety.
\begin{definition}[{\cite[4.5]{Jankowski2025}}]
    A (not necessarily normal) \textit{toric supervariety} is an irreducible supervariety $X$ equipped with a left action by a supertorus $T$, and a $T$-equivariant map $T \rightarrow X$ which factors through an open immersion of a homogeneous space $T/H$. We will assume $H_{\overline{0}}=1$. If $H =1$, then the toric supervariety will be called \textit{faithful}.
\end{definition}
\begin{lemma}[{\cite[Lemma 4.7]{Jankowski2025}}]
    Let $X = \mathrm{Spec}\,A$ be an affine supervariety. Then:
    \begin{enumerate}[label=(\alph*)]
        \item{ $X$ is toric (but not necessarily normal) for a quotient of the supertorus $T$ if and only if $A$ is isomorphic to a finitely-generated subalgebra of $\C[T]$ which is also a $\mathfrak{t}$-subrepresentation.}
        \item{ If $X$ is toric (but not necessarily normal) for the supertorus $T$, then it is faithful if and only if there is a non-nilpotent $T_{\overline{0}}$-weight vector $f\in A$ such that $A[f^{-1}] \cong\C[T ]$.}
    \end{enumerate}
\end{lemma}

Before we go ahead and verify the algebras arising from \cite{Kus-Fourier}, we consider a preparatory Lemma.

\begin{lemma}\label{prop-monomial-toric-super}
    Let $K\subseteq\{0,1\}^{q}\times\N^n$ be finite. Denote the set of sums of finitely many elements of $K$ with result in $\{0,1\}^{q}\times\N^n$ by $\langle K\rangle$. Assume that for all $(I,\mathbf{m})\in K$ we also have $(I-e_{i},\mathbf{m})\in K$ for all $1\leq i\leq q$ with $I_i=1$ and that we have
    \begin{equation*}
        \C[x^{\pm\mathbf{m}}v^{\pm}|(0,\mathbf{m})\in K]=\C[x_1^{\pm},\dots,x_{n}^{\pm},v^{\pm}].
    \end{equation*}
    Assume further that for all $1\leq i\leq q$ there exists $(e_i,\mathbf{m}^{i})\in\langle K\rangle$.
    Then, the superscheme $\mathrm{Spec}\, A$ with  
    \begin{equation*}
        A:=\C[\xi^{I}x^{\mathbf{m}}v\,|\,(I,\mathbf{m})\in K]
    \end{equation*}
    is a (not necessarily normal) toric supervariety. We can choose parameters $c_{ij}\in\C^{n}$ for the torus action subject to the conditions
    \begin{equation}\label{eq-prop-monomial-super-torus-action}
        \langle m,c_{ij}\rangle=0\text{ for all }(I,\mathbf{m})\in K\text{ with }(I+e_j,\mathbf{m})\notin\langle K\rangle\text{ and } I_j=0
    \end{equation}
    where $\langle\,,\rangle$ denotes the standard bilinear product.
\end{lemma}
\begin{proof}
    Define
    \begin{equation*}
        B:=\C[x^{\mathbf{m}}v|(0,\mathbf{m})\in K]
    \end{equation*}
    We have that
    \begin{equation*}
        \mathrm{Nil}(A)=\langle\xi^Ix^\mathbf{m}v|(I,\mathbf{m})\in K, |I|>0\rangle.
    \end{equation*}
    Hence $A/\mathrm{Nil}(A)\cong B$ implying that $\mathrm{Spec}\, A$ is irreducible.
    
    We also have the inclusion $B\subseteq A$. 
    Applying the localization from $B$ on $A$ yields us the supertorus $\C[x_1^{\pm},\dots,x_{n}^{\pm},v,\xi_1,\dots,\xi_q]$ via $\xi_i=\xi_{i}x^{\mathbf{m^{i}}}x^{\mathbf{-m^{i}}}$. This yields us that $A$ is GFR as the torus is an open subscheme of $\mathrm{Spec}\, A$.\smallskip
    
    Further, applying the derivation in the $i$-th direction to $\xi^{I}x^{\mathbf{m}}$ results in
    \begin{equation*}
        -\xi^{I-e_{i}}x^{\mathbf{m}}+\sum\limits_{j=1}^{q}\langle m,c_{ij}\rangle\xi^{I+e_{j}}x^{\mathbf{m}}=-\xi^{I-e_{i}}x^{\mathbf{m}}+\sum\limits_{\substack{j=1\\I_{j}=0}}^{q}\langle m,c_{ij}\rangle\xi^{I+e_{j}}x^{\mathbf{m}}
    \end{equation*}
    If $I_i=0$, then we get
    \begin{equation*}
        \sum\limits_{\substack{j=1\\I_{j}=0}}^{q}\langle m,c_{ij}\rangle\xi^{I+e_{j}}x^{\mathbf{m}}
    \end{equation*}
    Now, in order for $A$ to be closed under this derivation, we need the vanishing of $\langle m,c_{ij}\rangle$ whenever $\xi^{I+e_{j}}x^{\mathbf{m}}\notin A$ which is exactly the case described in \eqref{eq-prop-monomial-super-torus-action}.\smallskip

    Finally, we want to show that $A$ is reduced. Let $(I,\mathbf{m}))\in K$ with $1\leq i_1<\cdots<i_r\leq q$ such that $I_{i_{s}}=1$ for $1\leq s\leq r$. Then we have
    \begin{equation*}
        \xi^Ix^\mathbf{m}vv^{|I|}x^{\mathbf{m}^{i_{1}}}\cdots x^{\mathbf{m}^{i_{r}}}\in(A_{\overline{1}})
    \end{equation*}
    In the next paragraph, we are going to show that each $b\in B$ is a non-zero-divisor. This yields us that $\mathrm{Nil}(A)=J_A$.
    \smallskip
    
    As $A/\mathrm{Nil}(A)$ is an integral domain, we need to show that each $a\not\in\mathrm{Nil}(A)$ is a non-zero divisor. An element $a\in A\setminus\mathrm{Nil}(A)$ must be a linear combination of monomials of which at least one is of the form $x^{\mathbf{m}}v^{r}$. Now, choosing a monomial order such that $x_i>1>\xi_j$ allows one to use a leading term argument, concluding that such an $a$ is indeed a non-zero divisor. Hence, we can lift non-zero divisors of $\overline{A}$ to non-zero divisors of $A$ and thus $A$ is GFRR.\smallskip
\end{proof}
\begin{corollary}\label{Kus-Fourier-Toric}
    Let $\mathfrak{g}$ be a Lie superalgebra of type I, $\mathfrak{osp}(1|2n)$ or one of the basic exceptionals. Let $\lambda\in\mathfrak{h}^{*}$ be a dominant integral weight that is typical such that $k\lambda$ is also typical \cite[Theorem 1]{Kac_Rep} for all $k\in\N$. Denote by $P(\lambda)$ the corresponding lattice polytope from \cite{Kus-Fourier} and $S(\lambda)$ its lattice points.
    Then the superscheme $\mathrm{Spec}\, A$ for
    \begin{equation*}
        A:=\C[\xi^{I}x^{\mathbf{m}}v|(I,\mathbf{m})\in S(\lambda)]
    \end{equation*}
    is a faithful toric supervariety.
\end{corollary}
\begin{example}
Consider $\mathfrak{osp}(1|4)$ and the first fundamental weight $\lambda=\varpi_1$. 
Then the lattice points of the polytope using orthosymplectic Dyck path from \cite[Section 4.1]{Kus-Fourier} are the integer solutions to
\begin{equation*}
    s_{\alpha_1}+s_{\alpha_{1}+\alpha_2}+s_{\alpha_2}+s_{\delta_i}\leq 1,\,s_{\alpha_1}+s_{\alpha_{1}+\alpha_2}+s_{\alpha_{1}+\alpha_2+\alpha_1}+s_{\delta_1}\leq 1,\hspace{1em}i=1,2
\end{equation*}
with all variables being non-negative and $s_{\delta_i}\leq1$. The toric supervariety arising from this is the monomial superalgebra
\begin{equation*}
    \C[\xi_2v,\xi_2\xi_1v,\xi_2x_{\alpha_1+\alpha_2+\alpha_1}v,\xi_1v,v,x_{\alpha_{1}}v,x_{\alpha_{1}+\alpha_{2}}v,x_{\alpha_{2}}v,x_{\alpha_{1}+\alpha_{2}+\alpha_{1}}v,x_{\alpha_{2}}x_{\alpha_{1}+\alpha_{2}+\alpha_{1}}v]
\end{equation*}
\end{example}
\printbibliography
\end{document}